\documentclass[11pt]{amsart}

\usepackage{amsmath, amsbsy, amscd, amssymb, amsthm, ascmac, bm, enumerate, subcaption}
\usepackage{url}
\usepackage{hyperref,graphicx, xcolor}

\usepackage{etoolbox}
\patchcmd{\section}{\scshape}{\bfseries}{}{}
\makeatletter
\renewcommand{\@secnumfont}{\bfseries}
\makeatother

\theoremstyle{plain}
 \newtheorem{theorem}{Theorem}[section]

 \newtheorem{fact}[theorem]{Fact}
 \newtheorem{lemma}[theorem]{Lemma}
 
\theoremstyle{definition}
 \newtheorem{definition}[theorem]{Definition}
 \newtheorem{remark}[theorem]{Remark}

\renewcommand{\Re}{\operatorname{Re}}

\newcommand{\inner}[2]{\left\langle{#1},{#2}\right\rangle}

\numberwithin{equation}{section}
\numberwithin{figure}{section}
\numberwithin{table}{section}

\newcommand{\N}{\mathbb{N}}
\newcommand{\R}{\mathbb{R}}

\newcommand{\C}{\mathbb{C}}
\newcommand{\re}{\operatorname{Re}}
\newcommand{\im}{\operatorname{Im}}
\newenvironment{nouppercase}{%
  \renewcommand{\uppercasenonmath}[1]{}}{}
\title[Singly periodic maximal graphs in $\mathbb{L}^3$]
      {\LARGE Singly periodic maximal graphs with isolated singularities 
       in Lorentz-Minkowski 3-space}
\author[P.~Connor]
       {\large Peter Connor}
\address[P.~Connor]{Department of Mathematical Sciences \\
         Indiana University South Bend \\
         1700 Mishawaka Ave \\
         South Bend, IN 46634 \\ 
         USA}
\email{\href{mailto:pconnor@iu.edu}{pconnor@iu.edu}}
\urladdr{\url{https://pconnor.pages.iu.edu/}}
\author[S.~Fujimori]
       {\large Shoichi Fujimori}
\address[S.~Fujimori]{Department of Mathematics \\
         Hiroshima University \\
         Kagamiyama 1-3-1 \\
         Higashihiroshima \\ 
         Hiroshima 739-8526 \\ 
         Japan}
\email{\href{mailto:fujimori@hiroshima-u.ac.jp}{fujimori@hiroshima-u.ac.jp}}
\urladdr{\url{https://home.hiroshima-u.ac.jp/fujimori/}}
\date{January 24, 2026}
\thanks{
The second author was partially supported by JSPS Grant-in-Aid for Scientific Research (C) 25K06977.
}

\begin{document}

\begin{abstract}
Utilizing the Weierstrass representation for embedded doubly periodic minimal surfaces with parallel ends, we construct entire singly periodic graphs of spacelike maximal surfaces with isolated cone-like singularities 
in the Lorentz-Minkowski 3-space $\mathbb{L}^3$. 
\end{abstract}

\begin{nouppercase}
\maketitle
\section{Introduction}
A maximal surface in the Lorentz-Minkowski 3-space $\mathbb{L}^3$ is a spacelike surface with vanishing mean curvature. 
Like minimal surfaces in Euclidean 3-space $\mathbb{R}^3$, maximal surfaces in $\mathbb{L}^3$ have a Weierstrass-type representation. 

On the other hand, in contrast to the case of minimal surfaces in $\mathbb{R}^3$, it is known that the only complete maximal surfaces in $\mathbb{L}^3$ are spacelike planes \cite{Cal}, \cite{CY}. 
Therefore, it is natural to consider a class that allows for certain singularities in order to investigate the global properties of maximal surfaces. 
Maximal surfaces with singularities have interesting global properties different from those of minimal surfaces. 
For example, while Bernstein's theorem demonstrates that the only entire minimal graphs in $\mathbb{R}^3$ are planes, 
there are many entire maximal graphs with isolated singularities in $\mathbb{L}^3$. 
Maximal graphs with isolated singularities are investigated in \cite{FL}, \cite{K2}, \cite{LLS}, and so on. 
Among them, Fern\'andez, L\'opez, and  Souam \cite{FLS}, \cite{FLS2} consider the moduli space of maximal surfaces with isolated singularities. 
Although the dimension of the moduli space of these entire maximal graphs is investigated by \cite{FLS}, \cite{FLS2} and so on, only a few examples have been constructed previously. 

In this paper, we give a construction of entire singly periodic maximal graphs with isolated singularities by utilizing the construction of embedded doubly periodic minimal surfaces with parallel ends.
Our construction allows for making examples with any number of isolated singularities. 
Hence these maximal graphs are much more complicated than previously constructed examples. 
While not all doubly periodic minimal surfaces correspond to an entire singly periodic maximal graph, all of the examples constructed in \cite{CW} do that nicely because the doubly periodic minimal surfaces constructed in \cite{CW} look like horizontal planes connected by catenoidal necks. 
Each neck on one of these minimal surfaces corresponds to a cone-like singularity on the related maximal surface. 
For the relation between catenoid necks and cone-like singularities, see \cite{FRUYY2} for example. 
Thus we have an entire multi-graph connected by cone-like singularities as a corresponding doubly periodic maximal surface.  
Though the minimal surfaces are embedded, the corresponding maximal surfaces are not even immersed, because the waist of a catenoid neck in $\mathbb{R}^3$ corresponds to a single point in $\mathbb{L}^3$.  
Then after restricting our maximal surface appropriately, we obtain a singly periodic entire maximal graph with cone-like singularities. 
The process is aided by assuming two orthogonal symmetry planes for the doubly periodic minimal surfaces and one symmetry plane for the corresponding singly periodic maximal graphs. 
This allows for working with a Weierstrass representation that conveniently specifies the location of the necks in the case of the minimal surfaces and of the cone-like singularities in the case of the maximal surfaces.

The paper is organized as follows. 
In Section~\ref{sec:prelim}, we recall the basic properties of maximal surfaces with singularities based on 
\cite{UY}, \cite{FSUY}, and \cite{FRUYY2}. 
In Section~\ref{sec:2minperiodic}, we introduce the Weierstrass representation for embedded doubly periodic minimal surfaces with parallel ends given by \cite{CW}. 
Utilizing this Weierstrass representation, we give a construction of singly periodic entire maximal graphs in Section~\ref{sec:1periodic}, and give our main result, Theorem~\ref{thm:main}. 
Finally, we discuss the dimension of the moduli space and give examples of singly periodic entire maximal graphs in Section~\ref{sec:ex}. 

\section{Preliminaries}
\label{sec:prelim}

We denote by $\mathbb{L}^3$ the Lorentz-Minkowsiki 3-space with indefinite metric 
$\inner{~}{~}=dx_1^2+dx_2^2-dx_3^2$. 
Let $M$ be a Riemann surface. 
A conformal immersion $f:M\to\mathbb{L}^3$ is called 
a {\it spacelike surface} if the 
induced metric $ds^2=\inner{df}{df}$ is positive definite on $M$. 
A spacelike surface $f:M\to\mathbb{L}^3$ is called {\it maximal} if 
its mean curvature vanishes identically. 
In \cite{UY} a notion of {\it maxface} was introduced as a maximal surface with 
certain kind of singularities. 
More precisely, $f:M\to\mathbb{L}^3$ is called a {\it maxface} if there exists an 
open dense subset $W$ of $M$ such that the restriction $f|_W$ of $f$ to $W$ gives 
a conformal maximal immersion and $df(p)\ne 0$ for all $p\in M$. 

For maxfaces, a similar representation formula to the Weierstrass representation for 
minimal surfaces in $\mathbb{R}^3$ is known. 

\begin{theorem}[Weierstrass-type representation \cite{K, UY}]\label{th:w-type-rep}
Let $M$ be a Riemann surface and $f:M\to\mathbb{L}^3$ a maxface. 
Then there exist holomorphic 1-forms $\phi_1$, $\phi_2$, $\phi_3$ satisfying 
\begin{align}
\phi_1^2+\phi_2^2-\phi_3^2 &=0, \label{eq:confomal} \\
|\phi_1|^2+|\phi_2|^2+|\phi_3|^2 &>0, \label{eq:complete}
\end{align}
such that $f$ is represented by the following path integrals: 
\begin{equation}\label{eq:maxface}
f(z)=\Re\int_{z_0}^z \left(\phi_1,\,\phi_2,\,\phi_3\right), 
\end{equation}
where $z_0\in M$ is a fixed point. 
Conversely, let $\phi_1$, $\phi_2$, $\phi_3$ be holomorphic 1-forms on $M$ 
satisfying \eqref{eq:confomal} and \eqref{eq:complete}, then 
\eqref{eq:maxface} defines a maxface. 
\label{thm:maxface}
\end{theorem}

The induced metric $ds^2$ of $f$ is given by 
\[
ds^2=\frac{1}{2}\left(|\phi_1|^2+|\phi_2|^2-|\phi_3|^2\right). 
\]
The singular set $S(f)$ of $f$ is
\[
S(f)=\{p\in M\;;\;|\phi_1(p)|^2+|\phi_2(p)|^2-|\phi_3(p)|^2=0\}.
\]
We set 
\begin{equation}
G=\frac{-\phi_3}{\phi_1-i\phi_2}, 
\end{equation}
where $i=\sqrt{-1}$.  
Then $G$ is a meromorphic function on $M$, 
and $S(f)$ is written as 
\[
S(f)=\{p\in M\;;\;|G(p)|=1\}.
\] 
See \cite[Theorem 2.6]{UY}.
Moreover, 
$G|_{M\setminus S(f)}:M\setminus S(f)
\to(\mathbb{C}\cup\{\infty\})\setminus \{|z|=1\}$ 
coincides with the composition of the Gauss map 
\[
\nu|_{M\setminus S(f)}:M\setminus S(f)
\to H^2=\{x\in\mathbb{L}^3\;;\;\inner{x}{x}=-1\}
\]
of the maximal surface and the stereographic projection 
\[
\sigma:H^2\ni (x_1,x_2,x_3)
\mapsto\frac{x_1+ix_2}{1-x_3}\in\mathbb{C}\cup\{\infty\},
\]
that is, 
$G|_{M\setminus S(f)}=\sigma\circ \nu|_{M\setminus S(f)}$.
So we call $G$ the Gauss map of the maxface. 
Using this meromorphic function $G$ and the holomorphic 1-form $dh=-\phi_3$, 
we have   
\[
\phi_1=\frac{1}{2}\left(\frac{1}{G}+G\right)dh, 
\qquad
\phi_2=\frac{i}{2}\left(\frac{1}{G}-G\right)dh,  
\qquad
\phi_3=-dh,
\]
and 
\[
ds^2=\left(\frac{1}{|G|}-|G|\right)^2\frac{|dh|^2}{4}. 
\]

The triple $(M,\,G,\,dh)$ is called the {\it Weierstrass data} of $f$. 

\begin{remark}
The second fundamental form 
${\rm I}\!{\rm I}$ of the surface \eqref{eq:maxface} is given by
\[
{\rm I}\!{\rm I}=Q+\overline{Q},\qquad
Q=\frac{dGdh}{G}.
\]
$Q$ is called the {\em Hopf differential} of $f$.  
\end{remark}

Generic singularities of maxfaces are classified in \cite{FSUY}. 
Moreover, several criteria for singular points of maxfaces by using their 
Weierstrass data are given in \cite{FRUYY2, FSUY, UY}. 

\begin{definition}[Cone-like singular points \cite{FRUYY2}]\label{df:cone}
Let $f:M\to\mathbb{L}^3$ be a maxface with Weierstrass data $(M,G,dh)$. 
We denote by $S(f)$ the singular set of $f$, that is, 
$S(f)=\{p\in M\;;\;|G(p)|=1\}$.
\begin{enumerate}
\item A singular point $p\in S(f)$ of $f$ is called {\it non-degenerate} if 
$dG$ does not vanish at $p$. 
\item Let $S_0\subset S(f)$ be a connected component of $S(f)$ which consists 
of non-degenerate singular points. 
Then each point of $S_0$ is called a {\em generalized cone-like singular point} 
if $S_0$ is compact and the image $f(S_0)$ is one point.  
Moreover, if there is a neighborhood $U$ of $S_0$ and $f(U\setminus S_0)$ 
is embedded, 
each point of $S_0$ is called a {\em cone-like singular point}.
\end{enumerate}
\end{definition}

\begin{fact}{\cite[Lemmas 2.3 and 2.4]{FRUYY2}}
Let $f:M\to\mathbb{L}^3$ be a maxface with Weierstrass data $(M,G,dh)$,  
and $S_0\subset S(f)$ a connected component of $S(f)$. 
\begin{enumerate}
\item $S_0 $ consists of generalized cone-like singular points 
if and only if it is compact, and
\[
\frac{dG}{Gdh}\in\mathbb{R}\setminus\{0\}
\]
holds. 
\item Assume that $S_0$ consists of generalized cone-like singular points. 
Then it consists of cone-like singular points if and only if 
$G|_{S_0}: S_0\to S^1\subset\mathbb{C}$ is injective 
and $dh/G$ does not vanish on $S_0$.
\end{enumerate}
\end{fact}

\begin{definition}\label{df:isolate}
An \textit{isolated singularity} of a graph in $\mathbb{L}^3$ is an isolated point of a graph in $\mathbb{L}^3$ where the tangent plane cannot be defined and there exists a punctured neighborhood of this point in the graph such that the tangent plane is defined at any point in the neighborhood. 
Note that a cone-like singularity is a point in the underlying Riemann surface of the maximal surface, and the image of a cone-like singularity is an isolated singularity. 
\end{definition}

\section{Doubly periodic minimal surfaces}
\label{sec:2minperiodic}

There are many examples of embedded doubly periodic minimal surfaces with parallel ends that can be used to construct a singly periodic maximal graph.  Most known examples of doubly periodic minimal surfaces with parallel ends limit as a foliation of parallel planes with catenoid necks between successive planes that shrink to nodes at the limit.  Examples of this type of genus one were constructed by Karcher \cite{Ka2}, of genus two by Wei \cite{Wei}, and of arbitrary genus by the first author and Weber \cite{C2}, \cite{CW}.  See figure \ref{fg:CW} for two examples of genus eight.  Traditionally, doubly periodic minimal surfaces are oriented so that the ends are vertical and have horizontal normal vectors, in which case the period vectors are also horizontal.  The simplest examples have four ends in the quotient surface - two top ends and two bottom ends.  As shown in \cite{MR}, the ends are asymptotic to vertical half-planes, or flat annuli in the quotient.  The ends are referred to as Scherk ends.

\begin{figure}[htbp] 
\begin{center}
\begin{tabular}{ccc}
 \includegraphics[width=.4\linewidth]{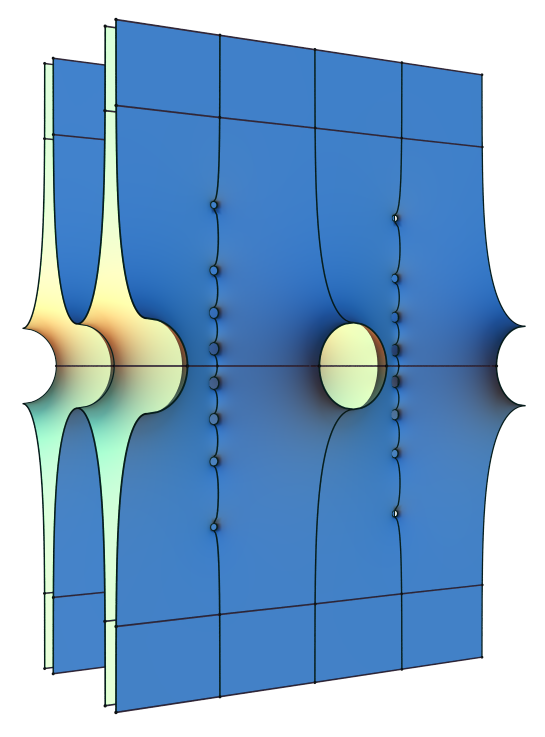} \hspace{1in}
 \includegraphics[width=.3\linewidth]{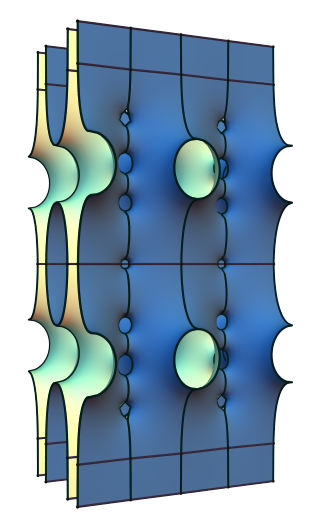} \\
\end{tabular}
\caption{Connor-Weber examples of genus eight doubly periodic minimal surfaces.}
\label{fg:CW}
\end{center}
\end{figure} 

Given a Riemann surface $M$, a meromorphic function $G$ and meromorphic one-form $dh$ on $M$, $(M,G,dh)$ is the Weierstrass data for a doubly periodic minimal surface with horizontal periods $T_1$ and $T_2$ and Scherk ends if the following hold:
\begin{enumerate}
\item
The zeros of $dh$ are the zeros and poles of $G$ on $M$ minus the ends, with the same multiplicity.
\item
$dh$ has a pole of order one and $G$ has finite value at each of the ends.
\item
For each closed curve $\gamma$ on $M$,
\[
\re\int_{\gamma}\left(\frac{1}{2}\left(\frac{1}{G}-G\right)dh,\frac{i}{2}\left(\frac{1}{G}+G\right)dh,dh\right)=(0,0,0)\mod\{T_1,T_2\}.
\]
This is called the {\it period problem}.
\end{enumerate}

Utilizing the observed geometry of the surfaces, one can come up with a possible Weierstrass data for a candidate doubly periodic minimal surface that behaves similarly to the Karcher, Wei, and Connor-Weber examples.  Assuming two orthogonal vertical symmetry planes, one can work with a fundamental piece consisting of one-fourth of the quotient surface.  See figure \ref{fg:CWfd}.  In this case, the fundamental piece has two ends.  Place the bottom end at $z=0$ and the top end at $z=\infty$, and orient the surface so the normal vectors at the ends of the fundamental piece point in the positive $x_1$ direction.  Then $G(0)=G(\infty)=1$.  Set the fundamental domain as $\{z\in\C^*\;;\;\im{z}\geq 0\}$, where $\C^*=\mathbb{C}\setminus\{0\}$.  The requirement that $dh$ has a pole of order one at the ends means we can set
\[
dh=\frac{1}{z}dz.
\]
Note that this means $dh$ has no zeros, and the first two requirements on the Weierstrass data are satisfied.  

\begin{figure}[htbp] 
\begin{center}
\begin{tabular}{ccc}
 \includegraphics[height=3in]{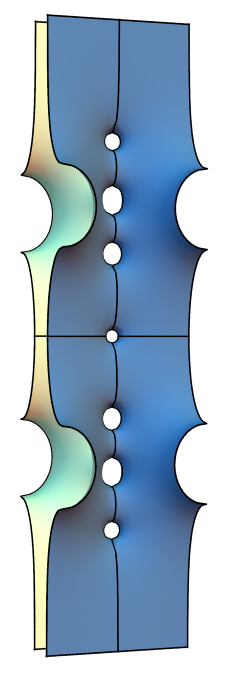} \hspace{1in}
 \includegraphics[height=3in]{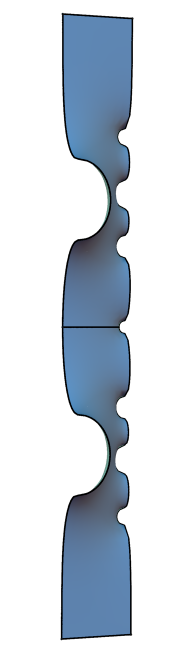} \\
\end{tabular}
\caption{Quotient (left) and fundamental piece (right) of a Connor-Weber genus eight doubly periodic minimal surface.}
\label{fg:CWfd}
\end{center}
\end{figure} 

The positive real axis will map to the right side of the fundamental piece, and the negative real axis will map to the left side of the fundamental piece.  
See the right-hand surface of Figure \ref{fg:CWfd}.
Assume the fundamental piece has $m$ handles opening in the negative $x_1$ direction (inward) and $n$ handles opening in the positive $x_1$ direction (outward).  
Place the $m$ handles that open inward on the right side of the fundamental piece, as in Figure \ref{fg:CWfd}.  
Ordering the inward pointing handles from bottom to top, for each $k$, let $a_{2k-1}$ and $a_{2k}$ map to the bottom and top, respectively, of the $k$-th handle opening inward.  
Place the $n$ handles that open outward on the left side of the fundamental piece, as in Figure \ref{fg:CWfd}.  Ordering these outward handles from bottom to top, for each $k$, let $b_{2k-1}$ and $b_{2k}$ map to the bottom and top, respectively, of the $k$-th handle opening outward.  So 
\[
b_{2n}<b_{2n-1}<\cdots<b_2<b_1<0<a_1<a_2<\cdots<a_{2m-1}<a_{2m}.
\]

Since the surface normal points in the positive $x_1$ direction at the ends of the fundamental piece, the surface normal points upward at the bottom and downward at the top of each handle opening inward, and the surface normal points upward at the bottom and downward at the top of each handle opening outward.  That is, $G(a_{2k-1})=\infty$ and $G(a_{2k})=0$ for $k=1,2,\ldots,m$, and $G(b_{2k-1})=0$ and $G(b_{2k})=\infty$ for $k=1,2,\ldots,n$.  
We can now define the map $G$.  Let 
\begin{equation}\label{eq:barM}
\overline{M}=\left\{(z,w)\in({\C}\cup\{\infty\})^2\;;\;w^2=\prod_{k=1}^m\frac{z-a_{2k}}{z-a_{2k-1}}\prod_{k=1}^n\frac{z-b_{2k-1}}{z-b_{2k}}\right\},
\end{equation}
\[
M=\overline{M}\setminus\{(0,\pm1),(\infty,\pm1)\},
\]
\[
G(z)=w,
\qquad
\text{and}
\qquad
dh=\frac{1}{z}dz.
\]
Assuming the period problem is solvable for an appropriate set of values of the $a_k$ and $b_k$ terms, we have the Weierstrass data $(M,G,dh)$ of a genus $m+n-1$ doubly periodic minimal surface with two top and top bottom ends in the quotient.  In order for $G(0)=1$, we need to set 
\[
b_{2n}=\prod_{k=1}^m\frac{a_{2k}}{a_{2k-1}}\prod_{k=1}^{n-1}\frac{b_{2k-1}}{b_{2k}}b_{2n-1}.
\]
Since we need $b_{2n}<b_{2n-1}$, this does place restrictions on the possible values of the $a_k$ and $b_k$ terms.

There are examples in which the handles on the right side of the fundamental piece do not all open inward (or outward on the left side).  If the handle between the images of $a_{2k-1}$ and $a_{2k}$ opens outward instead of inward then the normal at the bottom of the handle will point down and the normal at the top of the handle will point up.  That is, $G(a_{2k-1})=0$ and $G(a_{2k})=\infty$.  Similarly, if the handle between the images of $b_{2k-1}$ and $b_{2k}$ opens inward instead of outward then the normal at the bottom of the handle will point up and the normal at the top of the handle will point down.  That is, $G(b_{2k-1})=\infty$ and $G(b_{2k})=0$.  Taking into account this possibility, set

\[
w^2=\prod_{k=1}^m\left(\frac{z-a_{2k}}{z-a_{2k-1}}\right)^{\alpha_k}\prod_{k=1}^n\left(\frac{z-b_{2k-1}}{z-b_{2k}}\right)^{\beta_k}
\]
with $\alpha_k,\beta_k\in\{\pm 1\}$ for each $k\in\N$.

\section{Singly periodic maximal graphs}
\label{sec:1periodic}


Given the Weierstrass data $(M,G,dh)$ and representation
\[
(\omega_1,\omega_2,\omega_3)=\left(\frac{1}{2}\left(\frac{1}{w}-w\right)\frac{1}{z}dz,\frac{i}{2}\left(\frac{1}{w}+w\right)\frac{1}{z}dz,\frac{1}{z}dz\right)
\]
for a doubly periodic minimal surface developed in Section \ref{sec:2minperiodic}, we consider the corresponding maximal surfaces.  Since a maximal surface cannot have a horizontal normal and thus cannot have any vertical Scherk ends, we rotate the surface so that the ends are horizontal instead of vertical.  This results in the Weierstrass representation

\[
(\omega_1,\omega_2,\omega_3)=\left(\frac{i}{2}\left(\frac{1}{w}+w\right)\frac{1}{z}dz,\frac{1}{z}dz,\frac{1}{2}\left(\frac{1}{w}-w\right)\frac{1}{z}dz\right),
\]
with Weierstrass data
\[
G=\frac{i(w+1)}{w-1},\qquad dh=\frac{1}{2}\left(\frac{1}{w}-w\right)\frac{1}{z}dz
\]
for a doubly periodic minimal surface with four horizontal Scherk ends.  
Setting $(\phi_1,\phi_2,\phi_3)=(i\omega_1,i\omega_2,\omega_3)$, we get the Weierstrass representation 
\[
(\phi_1,\phi_2,\phi_3)=\left(-\frac{1}{2}\left(\frac{1}{w}+w\right)\frac{1}{z}dz, \frac{i}{z}dz, \frac{1}{2}\left(\frac{1}{w}-w\right)\frac{1}{z}dz\right)
\]
for a maximal surface, with Weierstrass data
\[
G=\frac{1+w}{1-w},\qquad
dh=-\frac{1}{2}\left(\frac{1}{w}-w\right)\frac{1}{z}dz.
\]
Note that $G(\infty)=\infty$, so the Scherk end at $z=\infty$ is horizontal.  The end at $z=0$ is not necessarily horizontal because
\[
G(0)=\frac{1+w(0)}{1-w(0)},
\]
so the end at $z=0$ is only horizontal if $w(0)=1$.  We are constructing a singly periodic surface, with the period of $(0,2\pi,0)$ given by traveling along any closed curve around either end.  Therefore, we do not need the ends to be horizontal.  However, after rotating if necessary, we can assume that one of the ends is horizontal.  Fix the end at $z=\infty$ to be horizontal, as provided by the given Weierstrass representation.

The following lemma is easy to prove and provides information about surface symmetries.
\begin{lemma}\label{lem:auto}
The Riemann surface $\overline{M}$ as in \eqref{eq:barM} has the automorphisms 
\[
\tau_1(z,w)=(\overline{z},\overline{w})
\qquad\text{and}\qquad
\tau_2(z,w)=(\overline{z},-\overline{w}).
\]
These transformations induce the following changes in the Weierstrass representation:
\[
\tau_1^*(\phi_1,\phi_2,\phi_3)=(\overline{\phi_1},-\overline{\phi_2},\overline{\phi_3}),
\qquad
\tau_2^*(\phi_1,\phi_2,\phi_3)=-(\overline{\phi_1},\overline{\phi_2},\overline{\phi_3}).
\]
\end{lemma}
The transformation $\tau_2$ demonstrates that there is a point reflection on the surface.  As we are constructing a graph, we do not want to utilize this.  Thus, we can restrict the domain to $w\geq 0$, in which case we can set
\begin{equation}\label{eq:m}
M=\C^*=\C\setminus\{0\}.
\end{equation}  
The transformation $\tau_1$ implies the surface has a vertical symmetry plane parallel to the $x_1x_3$-plane.  Therefore, the surface has fundamental domain 
\begin{equation}\label{eq:m1}
M_1=\{z\in\C^*:\im z\geq 0\}, 
\end{equation}
with the whole surface $f(M)$ obtained by taking the union of $f(M_1)$ with the reflection of $f(M_1)$ through that symmetry plane.

\begin{theorem}
Let $m,n\in\N$.  For each $j=1,2,\ldots ,m$ and $k=1,2,\ldots, n$, choose $a_j,b_k\in\R$ such that
\[
b_{2n}<b_{2n-1}<\cdots<b_2<b_1<0<a_1<a_2<\cdots<a_{2m-1}<a_{2m}.
\] 
For each $j=1,2,\ldots ,m$ and $k=1,2,\ldots, n$, choose $\alpha_k, \beta_j\in\{\pm1\}$.  
Define
\[
w^2=\prod_{k=1}^m\left(\frac{z-a_{2k}}{z-a_{2k-1}}\right)^{\alpha_k}\prod_{k=1}^n\left(\frac{z-b_{2k-1}}{z-b_{2k}}\right)^{\beta_k},
\]
\[
G=\frac{1+w}{1-w},
\qquad
\text{and} 
\qquad
dh=-\frac{1}{2}\left(\frac{1}{w}-w\right)\frac{1}{z}dz.
\]
Then $\left(\C^*,G,dh\right)$ is the Weierstrass data for a singly periodic maximal graph with two Scherk ends and $m+n$ cone-like singularities in the quotient with the following properties:
\begin{enumerate}
\item
The period at the ends is $\pm(0,2\pi,0)$.  There is no period problem at the cone-like singularities.
\item
For each $j=1,2,\ldots,m$ there is a cone-like singularity on $[a_{2j-1},a_{2j}]$ that points up if $\alpha_j=-1$ and points down if $\alpha_j=1$.
\item
For each $k=1,2,\ldots,n$ there is a cone-like singularity on $[b_{2k-1},b_{2k}]$ that points up if $\beta_k=1$ and points down if $\beta_k=-1$.
\item
As long as not all the cone-like singularities point in the same direction, it is possible to have both ends be horizontal.  
If all the cone-like singularities point up (resp. point down) then the end $z=0$ goes downward (resp. upward).  
\end{enumerate}
\label{thm:main}
\end{theorem}

The following three lemmas comprise the bulk of the proof of Theorem \ref{thm:main}.  Lemmas \ref{lem:conepts} and \ref{lem:conedir} prove properties (2) and (3) from Theorem \ref{thm:main}, and Lemma \ref{lem:embedded} proves that the surface defined in Theorem \ref{thm:main} is a graph.

\begin{lemma}
The intervals $[a_{2j-1},a_{2j}]$ and $[b_{2k-1},b_{2k}]$ consist of cone-like singular points for each $j=1,2,\ldots,m$ and $k=1,2,\ldots,n$. 
\label{lem:conepts}
\end{lemma}

\begin{proof}
By Definition \ref{df:cone}, $S(f)=\{z\in \C^*\;;\;|G(z)|=1\}$.  Note that we can define $G(z)=-1$ at the $m+n$ points $a_j$ and $b_k$ that are in the denominator of $w^2$, as $w^2(z)=\infty$ at those points.  Otherwise, observe that 
\[
|G|^2=G\overline{G}=\frac{1+w}{1-w}\frac{1+\overline{w}}{1-\overline{w}}=\frac{1+2 \re w+|w|^2}{1-2\re w+|w|^2}.
\]
Thus, $|G|=1$ if and only if
$\re w=0$.  Also, $\re w=0$ if and only if $w^2\leq 0$.  Hence, $|G|^2=1$ if and only if $w^2\leq 0$, and so $S(f)=\{z\in \C^*\;;\;w^2(z)\leq 0 \text{ or } w^2(z)=\infty\}$.  

Note that the degree of $w^2$ is $m+n$.  
If $x\in(b_{2l},b_{2l-1})$ for some $l\in\{1,2,\ldots, n\}$ then 
\begin{align*}
\left(\frac{x-a_{2j}}{x-a_{2j-1}}\right)^{a_j}&>0,\qquad j=1,2,\ldots,m, \\
\left(\frac{x-b_{2k-1}}{x-b_{2k}}\right)^{\beta_k}&>0,\qquad k=1,2,\ldots,n,\;\; k\neq l,
\end{align*}
and
\[
\left(\frac{x-b_{2l-1}}{x-b_{2l}}\right)^{\beta_l}<0.
\]
Hence, $w^2(x)<0$ if $x\in(b_{2l},b_{2l-1})$, and $w^2((b_{2l},b_{2l-1}))=(-\infty,0)$.  Similarly, $w^2(a_{2j-1},a_{2j})=(-\infty,0)$ for $j=1,2,\ldots, m$.  
Also, $w^2(z)=0$ only at the $m+n$ points $a_j$ and $b_k$ that are in the numerator of $w^2$.  Since $w^2$ has degree $m+n$, $w^2(z)\leq 0$ or $w^2(z)=\infty$ only when 
\[
z\in\left(\cup_{j=1}^m[a_{2j-1},a_{2j}]\right)\cup\left(\cup_{k=1}^n[b_{2k},b_{2k-1}]\right), 
\]
and $w^2|_{(a_{2j-1},a_{2j})}$ is one-to-one map onto $(-\infty,0)$ for $j=1,2,\ldots, m$ 
and $w^2|_{(b_{2k},b_{2k-1})}$ is one-to-one map onto $(-\infty,0)$ for $k=1,2,\ldots, n$.  

Hence,
\[
S(f)=\left(\cup_{j=1}^m[a_{2j-1},a_{2j}]\right)\cup\left(\cup_{k=1}^n[b_{2k},b_{2k-1}]\right).
\]
Since each connected component of $S(f)$ is a line segment, they can be parametrized as a regular curve in $M_1$.  Therefore, $S(f)$ consists of non-degenerate singular points. 


Note that if $x\in S(f)$ then $w^2(x)\leq0$ or $w^2(x)=\infty$ and so $w(x)\in i\R_{\geq0}\cup\{\infty\}$.  
Thus each $(x,w(x))\in S(f)$ is a fixed point under $\tau_2$ defined in Lemma~\ref{lem:auto}, and hence, $f$ is constant on each interval in $S(f)$, by Lemma~\ref{lem:auto}. 
That is, each interval $[b_{2k}, b_{2k-1}]$ and $[a_{2j-1},a_{2j}]$ consists of cone-like singular points for each $j=1,2,\ldots, m$ and $k=1,2,\ldots, n$.
\end{proof}

%
%
\begin{lemma}
For each $j=1,2,\ldots,m$, the cone-like singularity on the interval $[a_{2j-1},a_{2j}]$ points up if $\alpha_j=1$ and points down if $\alpha_j=-1$.  For each $k=1,2,\ldots,n$, the cone-like singularity on the interval $[b_{2k-1},b_{2k}]$ points down if $\beta_k=1$ and points up if $\beta_k=-1$.
\label{lem:conedir}
\end{lemma}

\begin{proof}
If a cone-like singularity points up then, going from right to left at the cone-like singularity, the normal goes from parallel to $(-1,0,1)$ to parallel to $(1,0,1)$.  If a cone-like singularity points down then, going from right to left at the cone-like singularity, the normal goes from parallel to $(1,0,1)$ to parallel to $(-1,0,1)$.  

The Gauss map is $\nu=\sigma^{-1}\circ G$, the composition of the inverse of stereographic projection with $G$.  Hence,
\[
\frac{\nu}{|\nu|_E}=\frac{(-2\re(G(z)),-2\im(G(z)),|G(z)|^2+1)}{\sqrt{|G(z)|^4+6|G(z)|^2+1}}, 
\]
where $|\nu|_E$ gives the Euclidean norm of $\nu$. 
Thus, if $G(z)=1$ (resp. $G(z)=-1$) then 
$$
\frac{\nu}{|\nu|_E}=\left(-\frac{1}{\sqrt{2}},0,\frac{1}{\sqrt{2}}\right)
\qquad
\left(\text{resp. }\frac{\nu}{|\nu|_E}=\left(\frac{1}{\sqrt{2}},0,\frac{1}{\sqrt{2}}\right)\right).
$$
Since $G(a_{2j-1})=-\alpha_j$ and $G(a_{2j})=\alpha_j$, the cone-like singularity on the interval $[a_{2j-1},a_{2j}]$ points up if $\alpha_j=-1$ and points down if $\alpha_j=1$.  Since $G(b_{2k-1})=\beta_k$ and $G(b_{2k})=-\beta_k$, the cone-like singularity on the interval $[b_{2k-1},b_{2k}]$ points up if $\beta_k=1$ and points down if $\beta_k=-1$. 
\end{proof}

\begin{lemma}\label{lem:embedded}
$f(M)$ is a graph over the $x_1x_2$-plane, 
where $M$ is as in \eqref{eq:m}.
\end{lemma}

\begin{proof}
We show that $f=(f_1,f_2,f_3)$ is a graph by showing that $f(M_1)$ is a graph over the $x_1x_2$-plane, 
where $M_1$ as in \eqref{eq:m1} is the fundamental domain for our surface.  Similar to what was done in the proof of Lemma \ref{lem:conepts}, $|G|\geq 1$ if and only if $\re w\geq 0$, which holds since $w$ is a square root.  Hence, $|G(z)|\geq 1$ for all $z\in M_1$, and $|G(z)|>1$ for all $z\in M_1\setminus S(f)$.  Therefore, the Gauss map always points up on $f(M_1)$.

Since $w\in i\R_{\geq 0}\cup\{\infty\}$ on $S(f)$, $f$ is constant on each interval $[a_{2j-1},a_{2j}]$ and $[b_{2j},b_{2j-1}]$ for $j=1,2,\ldots,m$ and $k=1,2,\ldots,n$.  If $z\in\R_{>0}\setminus (S(f)\cup\{0\})$ then $w\in\R_{>0}$, in which case $f_1$ is decreasing there.  If $z\in\R_{<0}\setminus (S(f)\cup\{0\})$ then $w\in\R_{>0}$, in which case $f_1$ is increasing there.  Hence, the projections of the images of the negative and positive real axes onto the $x_1x_2$-plane are embedded.  Note that since $f_2(z)=-\text{Arg}(z)$, the projections of the images of the negative and positive real axes will lie in lines parallel to the $x_2$-axis that are a distance $\pi$ apart from each other.  Since the Gauss map always points up on the surface, the projection of $f(M_1)$ onto the $x_1x_2$-plane is a submersion.  Hence, $f(M)$ is a graph over the $x_1x_2$-plane. 
\end{proof}

Now we can prove Theorem~\ref{thm:main}.
 
\begin{proof}[Proof of Theorem~\ref{thm:main}]
By Theorem \ref{thm:maxface} the Weierstrass representation from Theorem \ref{thm:main} is the Weierstrass representation of a maxface.  Lemma \ref{lem:embedded} proves that the surface defined in Theorem \ref{thm:main} is a graph.

The periods at the ends can be computed as residues.  Since the residues of $\phi_1$ and $\phi_3$ are real at $z=0$ and $z=\infty$, the residue of $\phi_2$ at $z=0$ is $i$, and the residue of $\phi_2$ at $z=\infty$ is $-i$, we see that
\[
\re\int_{C(0)}(\phi_1,\phi_2,\phi_3)=(0,-2\pi,0) 
\]
and
\[
\re\int_{C(\infty)}(\phi_1,\phi_2,\phi_3)=(0,2\pi,0), 
\]
where $C(z)$ is a closed curve that surrounds $z$ counterclockwise.
As there are no other periods, property (1) from Theorem \ref{thm:main} is proven.

Lemmas \ref{lem:conepts} and \ref{lem:conedir} prove properties (2) and (3) from Theorem \ref{thm:main}.

Property (4) from Theorem \ref{thm:main} is a consequence of the maximum principle.
\end{proof}

\section{Examples}
\label{sec:ex}

Examples of singly periodic maximal graphs with cone-like singularities derived from Theorem \ref{thm:main} split into categories based on the number of cone-like singularities along the positive and negative real axes.  There will always be at least one cone.  We assume there is always at least one cone along the positive real axis.  A singly periodic maximal graph of type $(m,n)$ will have $m\geq 1$ cones along the positive real axis and $n\geq 0$ cones along the negative real axis.  Utilizing reflections if necessary, we can assume that $m\geq n$.  Each of the $m+n$ cones can point up or down.  After reflecting if necessary, assume that the cone corresponding to the interval $(a_1,a_2)$ points up.  Otherwise, there are $2^{m+n-1}$ choices for the direction of the remaining $m+n-1$ cones.  Hence, there are up to $2^{m+n-1}$ distinct types of maximal graphs of type $(m,n)$.  We will see that there are fewer for some types due to some types being equivalent up to rotation/reflection.

As shown in \cite{FLS2}, if we remove the vertical symmetry plane then the moduli space for these surfaces has dimension $3(m+n)+1$.  With the extra assumptions of a vertical symmetry plane and placing the cone-like singularities along the images of the positive and negative real axes, the dimension of the moduli space of our examples is smaller.  There are $2(m+n)$ real parameters in the Weierstrass representation.  However, we only fixed the location of two points, placing the ends at $0$ and $\infty$.  Therefore we can fix the location of one of the other parameters.  Hence, the dimension of the moduli space of the surfaces derived from Theorem \ref{thm:main} is $2(m+n)-1$.  The $2(m+n)-1$ parameters primarily control the location along the image of the positive or negative real axes and size of the cone-like singularities and also control the angle of the end at $z=0$.


\subsection{Examples with four cone-like singularities}
There are seventeen different surfaces with four cone-like singularities given by Theorem \ref{thm:main}.  There are six surfaces of type $(4,0)$.  In this case, all four cones are along the image of the positive real axis.  See Figure \ref{fig:(4,0)}.
Note that in all figures of maximal graphs in this section, the vertical direction indicates the timelike axis.  
\begin{figure}[t]
  \centering
  \begin{subfigure}[b]{0.3\textwidth}
    \centering
    \includegraphics[width=\textwidth]{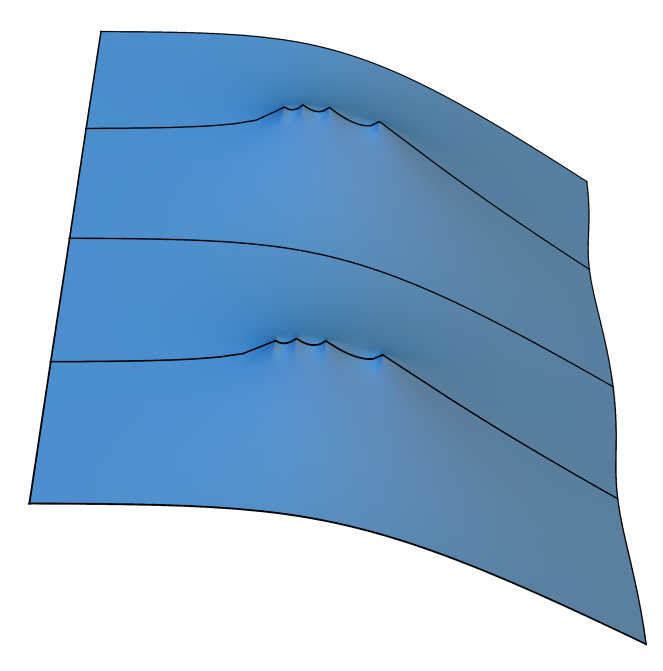}
    \caption*{}
  \end{subfigure} \hspace{.1in}
  ~ 
  \begin{subfigure}[b]{0.3\textwidth}
    \centering
    \includegraphics[width=\textwidth]{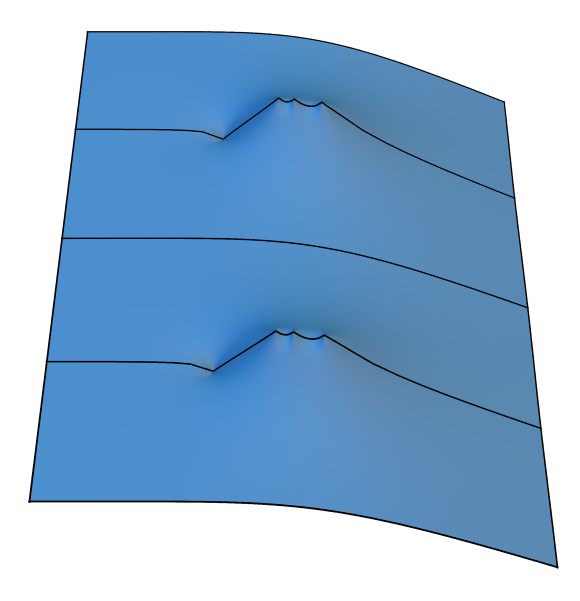}
    \caption*{}
  \end{subfigure} \hspace{.1in}
  ~
  \begin{subfigure}[b]{0.3\textwidth}
    \centering
    \includegraphics[width=\textwidth]{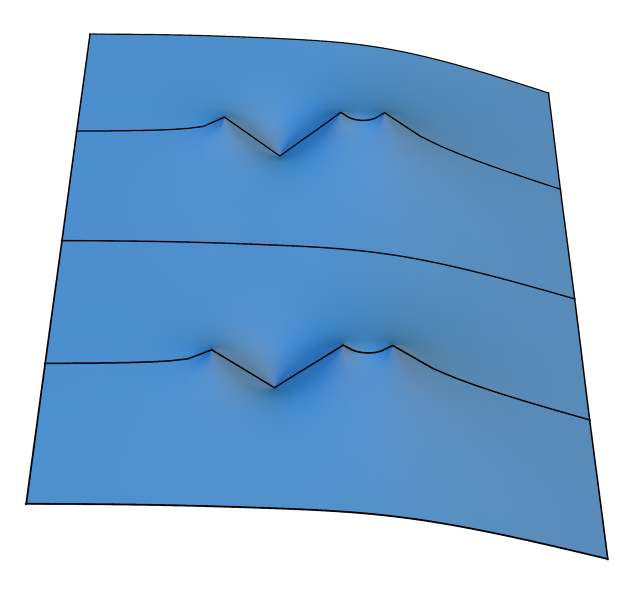}
    \caption*{}
  \end{subfigure}
  \begin{subfigure}[b]{0.3\textwidth}
    \centering
    \includegraphics[width=\textwidth]{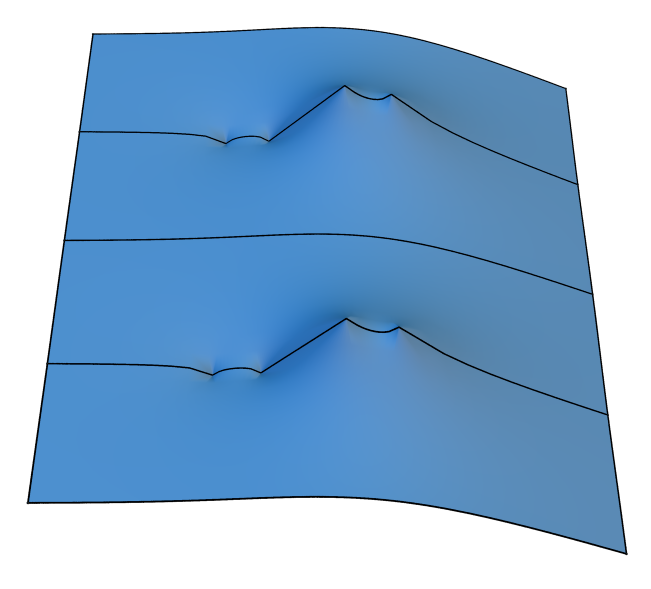}
  \end{subfigure} \hspace{.1in}
  ~ 
  \begin{subfigure}[b]{0.3\textwidth}
    \centering
    \includegraphics[width=\textwidth]{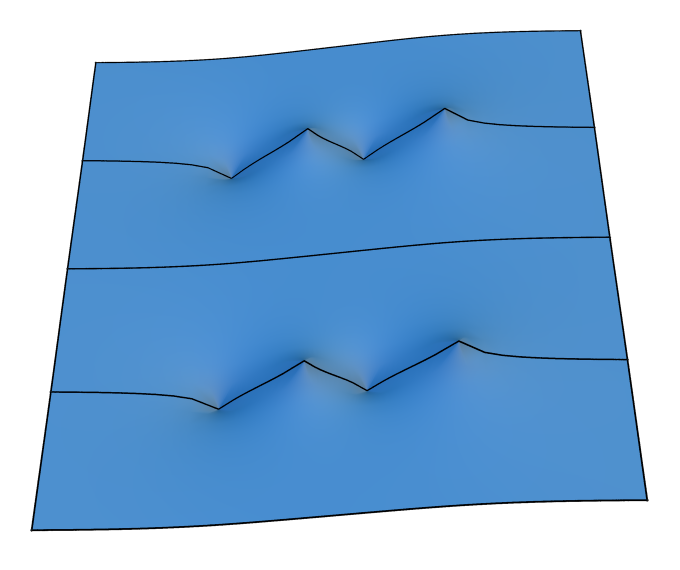}
  \end{subfigure}\hspace{.1in}
  ~
  \begin{subfigure}[b]{0.3\textwidth}
    \centering
    \includegraphics[width=\textwidth]{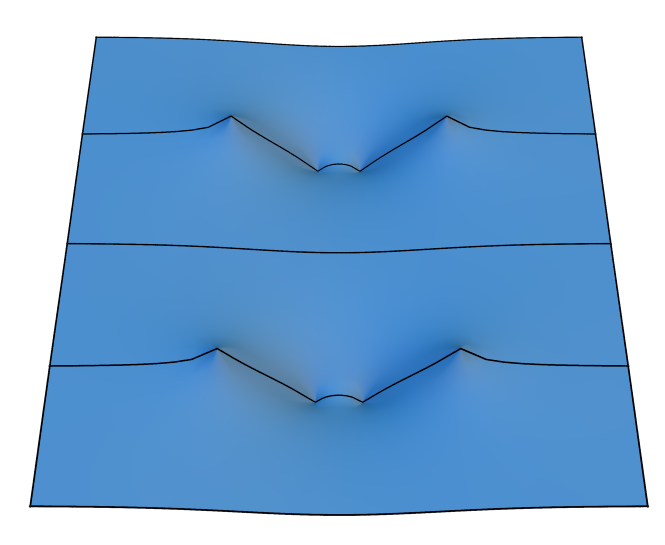}
  \end{subfigure}
  \caption{Singly periodic maximal graphs of type $(4,0)$.}
  \label{fig:(4,0)}
\end{figure}

There are six surfaces of type $(3,1)$, with three cones along the image of the positive real axis and one cone along the image of the negative real axis.  See Figure \ref{fig:(3,1)}.

\begin{figure}[t]
  \centering
  \begin{subfigure}[b]{0.3\textwidth}
    \centering
    \includegraphics[width=\textwidth]{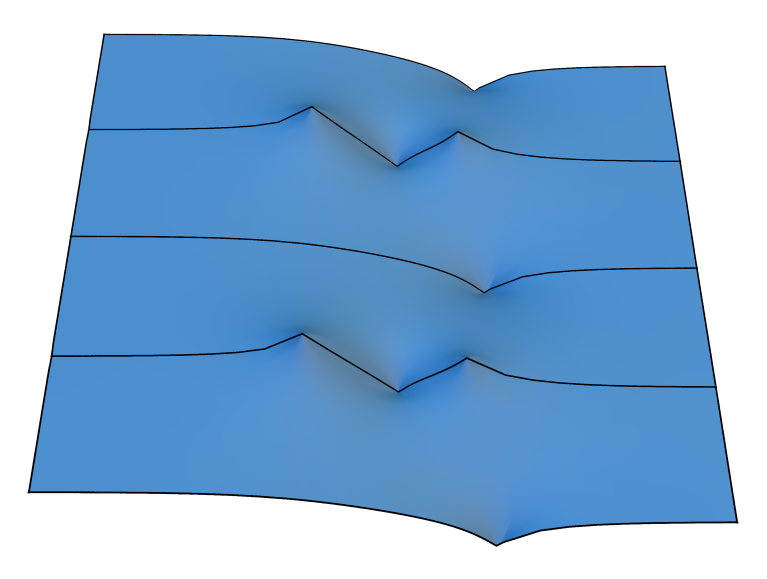}
    \caption*{}
  \end{subfigure} \hspace{.1in}
  ~ 
  \begin{subfigure}[b]{0.3\textwidth}
    \centering
    \includegraphics[width=\textwidth]{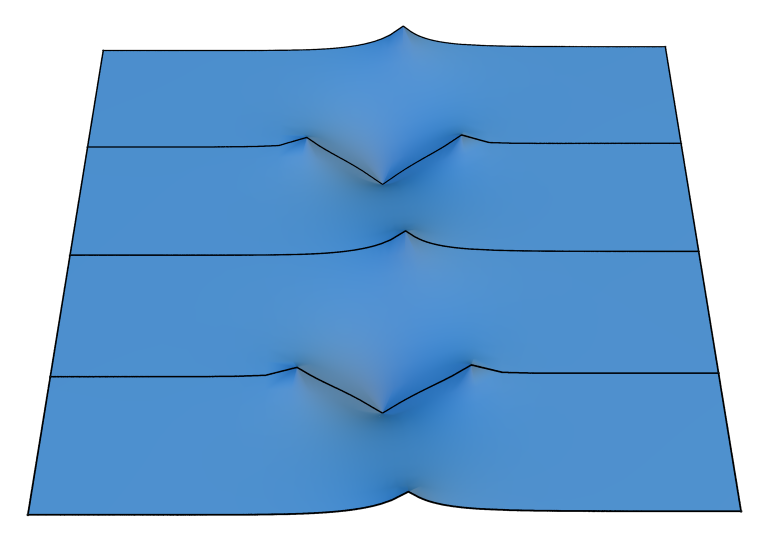}
    \caption*{}
  \end{subfigure} \hspace{.1in}
  ~
  \begin{subfigure}[b]{0.3\textwidth}
    \centering
    \includegraphics[width=\textwidth]{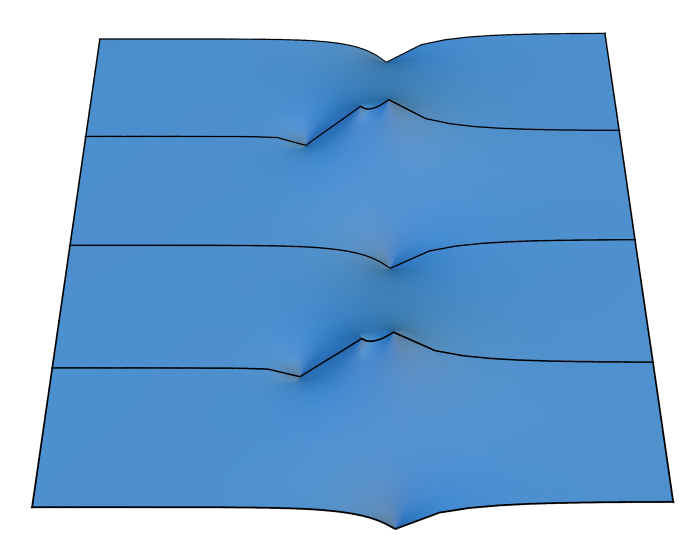}
    \caption*{}
  \end{subfigure}
  \begin{subfigure}[b]{0.3\textwidth}
    \centering
    \includegraphics[width=\textwidth]{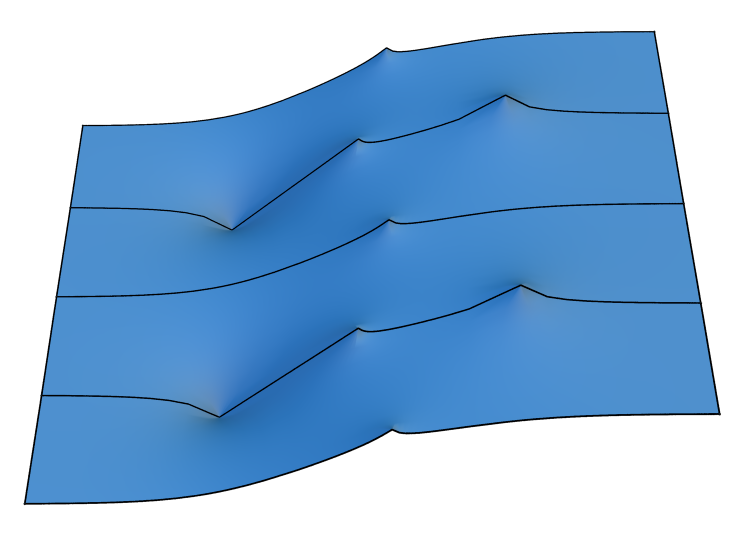}
    \caption*{}
  \end{subfigure} \hspace{.1in}
  ~ 
  \begin{subfigure}[b]{0.3\textwidth}
    \centering
    \includegraphics[width=\textwidth]{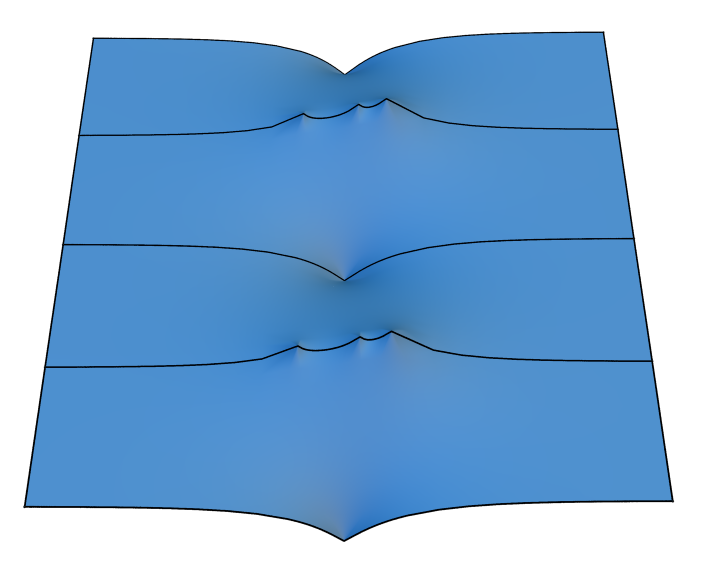}
  \end{subfigure} \hspace{.1in}
  ~
  \begin{subfigure}[b]{0.3\textwidth}
    \centering
    \includegraphics[height=1.3in]{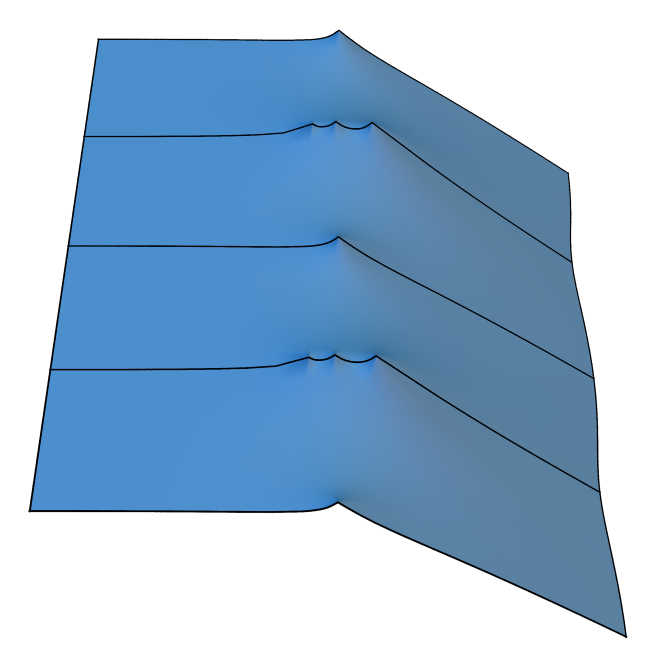}
  \end{subfigure}
  \caption{Singly periodic maximal graphs of type $(3,1)$.}
  \label{fig:(3,1)}
\end{figure}

There are five surfaces of type $(2,2)$, with two cones along the image of the positive real axis and two cones along the image of the negative real axis.  See Figure \ref{fig:(3,1)}.

\begin{figure}[t]
  \centering
  \begin{subfigure}[b]{0.3\textwidth}
    \centering
    \includegraphics[width=\textwidth]{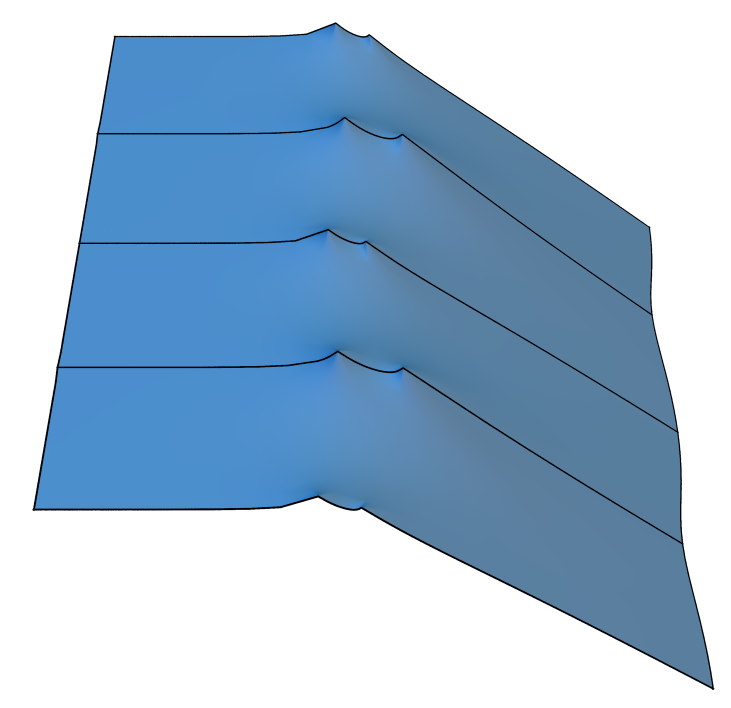}
    \caption*{}
  \end{subfigure} \hspace{.1in}
  ~ 
  \begin{subfigure}[b]{0.3\textwidth}
    \centering
    \includegraphics[width=\textwidth]{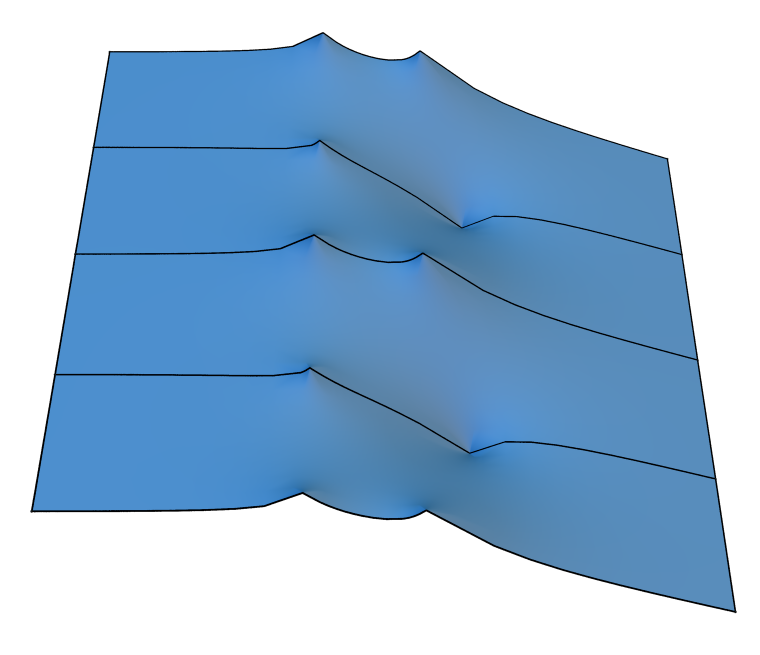}
    \caption*{}
  \end{subfigure} \hspace{.1in}
  ~
  \begin{subfigure}[b]{0.3\textwidth}
    \centering
    \includegraphics[width=\textwidth]{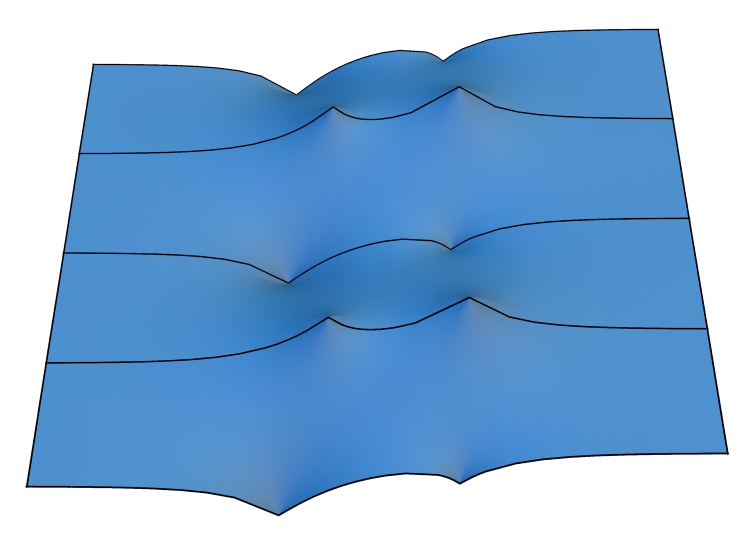}
    \caption*{}
  \end{subfigure} \hspace{.1in}
  \begin{subfigure}[b]{0.3\textwidth}
    \centering
    \includegraphics[width=\textwidth]{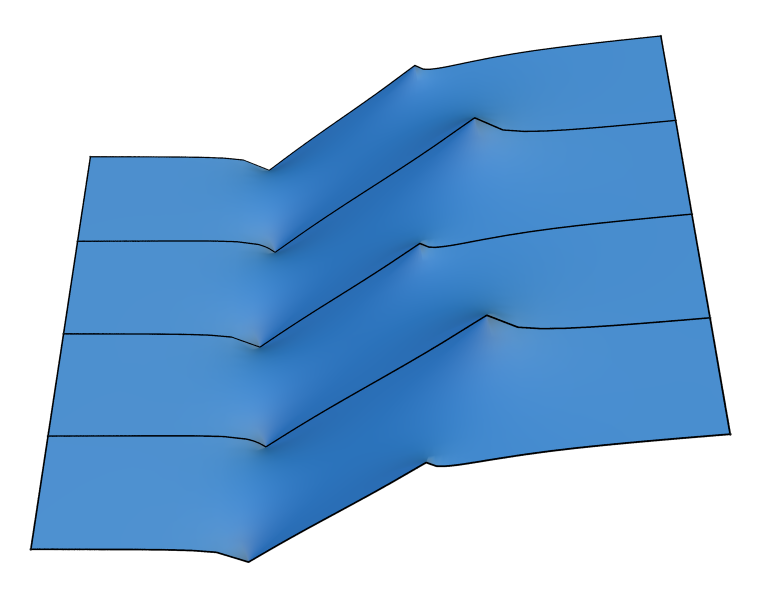}
  \end{subfigure} 
  ~ 
  \begin{subfigure}[b]{0.33\textwidth}
    \centering
    \includegraphics[width=\textwidth]{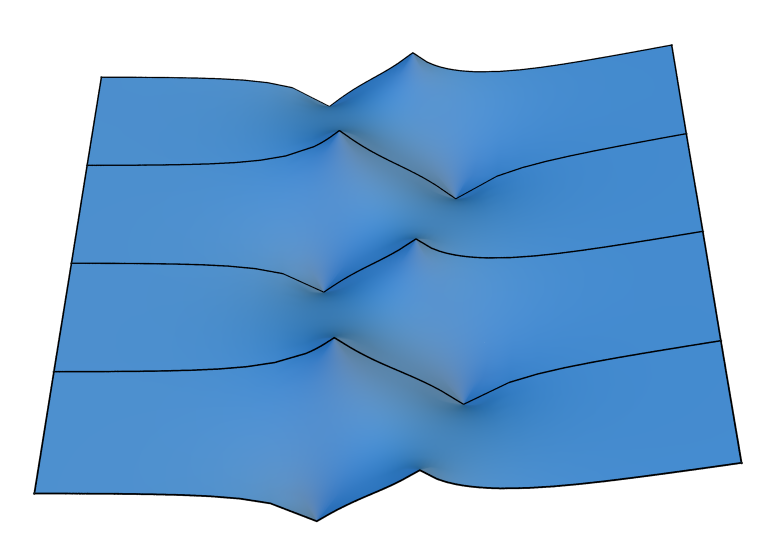}
  \end{subfigure}
  \caption{Singly periodic maximal graphs of type $(2,2)$.}
  \label{fig:(2,0)}
\end{figure}

\subsection{Examples with nine cone-like singularities}
There are five different types of singly periodic maximal graphs with nine cone-like singularities, with each type having up to $2^8$ variations on the directions of the cones.  The minimal surfaces in Figure \ref{fg:CW} correspond to singly periodic maximal graphs with nine cone-like singularities of types $(8,1)$ and $(7,2)$.  See Figure \ref{fig:9cones}

\begin{figure}[t]
  \centering
  \begin{subfigure}[b]{0.47\textwidth}
    \centering
    \includegraphics[height=1in]{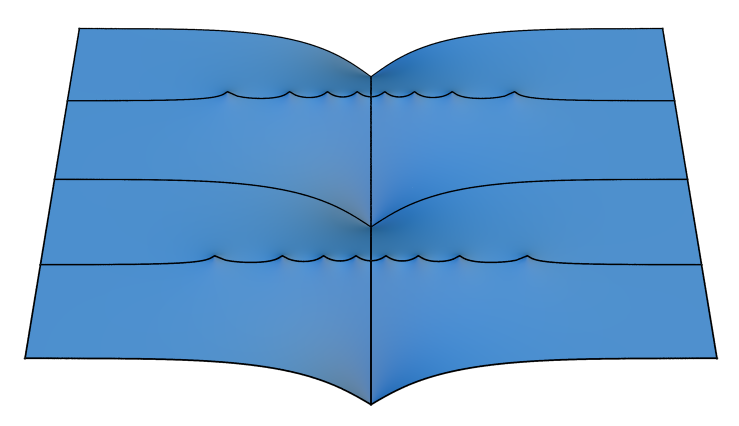}
    \caption*{Type $(8,1)$}
  \end{subfigure} \hspace{.1in}
  ~ 
  \begin{subfigure}[b]{0.47\textwidth}
    \centering
    \includegraphics[height=1in]{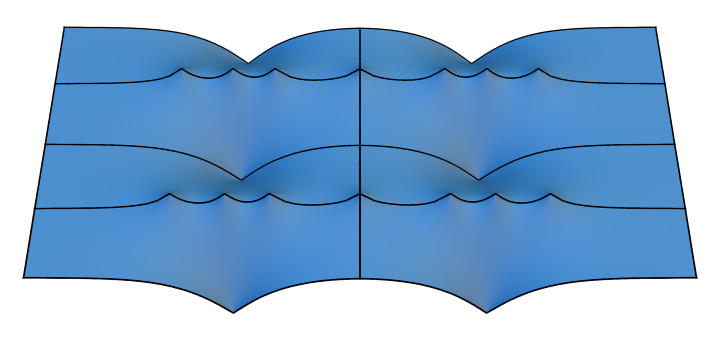}
    \caption*{Type $(7,2)$}
  \end{subfigure} \hspace{.1in}
  \caption{Singly periodic maximal graphs with nine cone-like singularities.}
  \label{fig:9cones}
\end{figure}

\newpage


\end{nouppercase}
\end{document}